\documentclass[10pt,reqno]{amsart}
\usepackage{graphicx,amsmath,amssymb}
\newtheorem{thm}{Theorem}[section]

\newtheorem{lem}[thm]{Lemma}

\theoremstyle{definition}

\theoremstyle{remark}
\newtheorem{rk}[thm]{Remark}

\numberwithin{equation}{section}

\begin{document}

\title[Inhomogeneous Strichartz estimates]{Inhomogeneous Strichartz estimates for Schr\"odinger's equation}

\author{Youngwoo Koh and Ihyeok Seo}

\address{School of Mathematics, Korea Institute for Advanced Study, Seoul 130-722, Republic of Korea}
\email{ywkoh@kias.re.kr}

\address{Department of Mathematics, Sungkyunkwan University, Suwon 440-746, Republic of Korea}
\email{ihseo@skku.edu}

\thanks{Ihyeok Seo was supported by the TJ Park Science Fellowship of POSCO TJ Park Foundation}

\begin{abstract}
Foschi and Vilela in their independent works (\cite{F},\cite{V}) showed that
the range of $(1/r,1/\widetilde{r})$ for which the inhomogeneous Strichartz estimate
$
\big\|\int_{0}^{t}e^{i(t-s)\Delta}F(\cdot,s)ds\big\|_{L^{q}_tL^{r}_x} \lesssim \|F\|_{L^{\widetilde{q}'}_tL^{\widetilde{r}'}_x}
$
holds for some $q,\widetilde{q}$
is contained in the closed pentagon with vertices $A,B,B',P,P'$ except the points $P,P'$
(see Figure \ref{figure}).
We obtain the estimate for the corner points $P,P'$.
\end{abstract}
\maketitle
% ----------------------------------------------------------------
\section{Introduction}

In this paper we consider the following Cauchy problem for the Schr\"odinger equation:
    $$
    \begin{cases}
    i\partial_{t}u+\Delta u=F(x,t),\\
    u(x,0)=f(x),
    \end{cases}
    $$
where $(x,t)\in\mathbb{R}^n\times\mathbb{R}$, $n\geq1$.
By Duhamel's principle, we have the solution
    $$
    u(x,t)=e^{it\Delta}f(x)-i\int_{0}^{t}e^{i(t-s)\Delta}F(\cdot,s)ds,
    $$
where $e^{it\Delta}$ is the free Schr\"odinger propagator defined by
$$e^{it\Delta}f(x)=(2\pi)^{-n}\int_{\mathbb{R}^n}e^{ix\cdot\xi-it|\xi|^{2}}\widehat{f}(\xi)d\xi.$$

The Strichartz estimates for the solution play important roles
in the study of well-posedness for nonlinear Schr\"odinger equations ({\it cf. \cite{C,T}}).
They actually consist of two parts, homogeneous $(F=0)$ and inhomogeneous $(f=0)$ part.
The homogeneous Strichartz estimate
    $$
    \|e^{it\Delta}f\|_{L^q_tL^r_x} \lesssim \|f\|_{L^2}
    $$
holds if and only if $(r,q)$ is admissible pair, that is,
    $$
    r,q \geq 2,\quad(n,r,q)\neq(2,\infty,2)\quad\text{and}\quad n/r+2/q=n/2
    $$
(see \cite{St,GV,M,KT} and references therein).
But determining the optimal range of $(r,q)$ and $(\widetilde{r},\widetilde{q})$ for which
the inhomogeneous Strichartz estimate
\begin{equation}\label{1:44}
\bigg\|\int_{0}^{t}e^{i(t-s)\Delta}F(\cdot,s)ds\bigg\|_{L^{q}_tL^{r}_x} \lesssim \|F\|_{L^{\widetilde{q}'}_tL^{\widetilde{r}'}_x}
\end{equation}
holds is not completed yet when $n\geq3$.
It was observed that this estimate is valid on a wider range than what is given
by admissible pairs $(r,q)$, $(\widetilde{r},\widetilde{q})$ (see \cite{CW}, \cite{K}).
Foschi and Vilela in their independent works (\cite{F},\cite{V}) showed that
the range of $(1/r,1/\widetilde{r})$ for which \eqref{1:44} is valid for some $q,\widetilde{q}$
is contained in the closed pentagon with vertices $A,B,B',P,P'$ except the points $P,P'$
(see Figure \ref{figure}).
The aim of this paper is to obtain \eqref{1:44} for the points $P,P'$.
Our result is the following.

%%%%%%%%%%%%%%%%%%%%%%%%%%%%%%%%%%%%%%%%%%%%%%%%%%%%%%%%%%%%%%%%%%%%%%%%%%%%%%%%%%%%%%%%%%%%%%%%%%%%%%%%%%
\begin{figure}[t!]\label{figure}
\includegraphics[width=7.0cm]{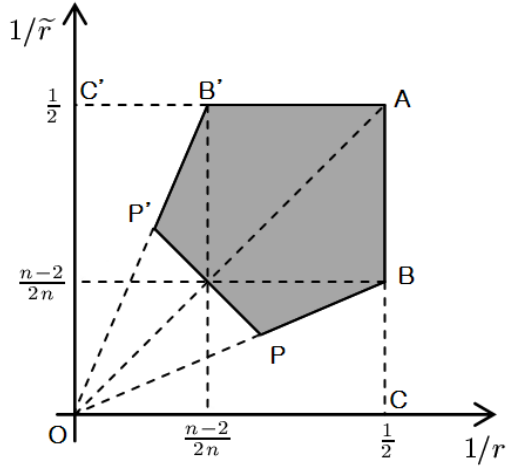}
\caption{The range of $(1/r,1/\widetilde{r})$ for \eqref{1:44} when $n\geq3$.
Here $P=(\frac{n-2}{2(n-1)},\frac{(n-2)^2}{2n(n-1)})$ and $P'=(\frac{(n-2)^2}{2n(n-1)},\frac{n-2}{2(n-1)})$.}
\end{figure}
%%%%%%%%%%%%%%%%%%%%%%%%%%%%%%%%%%%%%%%%%%%%%%%%%%%%%%%%%%%%%%%%%%%%%%%%%%%%%%%%%%%%%%%%%%%%%%%%%%%%%%%%%%

\begin{thm}\label{main_thm}
Let $n\geq3$. Then \eqref{1:44} holds when
$$(\frac{1}{r},\frac{1}{\widetilde{r}})=P=(\frac{n-2}{2(n-1)},\frac{(n-2)^2}{2n(n-1)})
\quad\text{if}\quad
\frac{n-2}{2(n-1)} \leq \frac{1}{q} = \frac{1}{\widetilde{q}'} < \frac{n}{2(n-1)},$$
and when
$$(\frac{1}{r},\frac{1}{\widetilde{r}})=P'=(\frac{(n-2)^2}{2n(n-1)},\frac{n-2}{2(n-1)})
\quad\text{if}\quad
\frac{n-2}{2(n-1)} < \frac{1}{q} = \frac{1}{\widetilde{q}'} \leq \frac{n}{2(n-1)}.$$
\end{thm}

\begin{rk}
Since $1/r+1/\widetilde{r}=(n-2)/n$, the condition $q=\widetilde{q}'$ follows from
the scaling condition
\begin{equation}\label{sca}
\frac1q+\frac1{\widetilde{q}}+\frac n2(\frac1r+\frac1{\widetilde{r}})=\frac n2.
\end{equation}
The conditions $1/q<n/2(n-1)$ and $(n-2)/2(n-1)<1/\widetilde{q}'$
when $(1/r,1/\widetilde{r})=P\,\, \text{and}\,\, P'$ correspond to
the known necessary conditions (\cite{F,V})
    $$
    \frac1q<\frac n2(1-\frac2r)\quad\text{and}\quad\frac1q<\frac n2(1-\frac2r),
    $$
respectively.
In Section \ref{sec3} we also give new necessary conditions for \eqref{1:44}.
\end{rk}

\begin{rk}
Our proof can be easily modified to cover the range of $(r,q)$ and $(\widetilde{r},\widetilde{q})$ obtained by Foschi and Vilela.
But we have chosen to present the proof only for the points $P,P'$ to keep the exposition as simple as possible.
The case $n\geq4$ in Theorem \ref{main_thm} was already shown in \cite{Ko} but the argument there does not suffice to obtain the same result in dimension $n=3$.
\end{rk}

\section{Proof of Theorem \ref{main_thm}}\label{sec2}

Under the same conditions in Theorem \ref{main_thm}, we will show
    \begin{equation}\label{1:7}
    \bigg\|\int_{-\infty}^{t}e^{i(t-s)\Delta}F(\cdot,s)ds\bigg\|_{L^{q}_tL^{r}_x}
    \lesssim \|F\|_{L^{\widetilde{q}'}_tL^{\widetilde{r}'}_x}
    \end{equation}
which implies \eqref{1:44}.
Indeed, to obtain \eqref{1:44} from \eqref{1:7},
first decompose the $L_t^q$ norm in the left-hand side of \eqref{1:44}
into two parts, $t\geq0$ and $t<0$. Then the latter can be reduced to the former
by changing the variable $t\mapsto-t$, and so it is only needed to consider the first part $t\geq0$.
But, since $[0,t)=(-\infty,t)\cap[0,\infty)$, applying \eqref{1:7} with $F$ replaced by $\chi_{[0,\infty)}(s)F$,
one can bound the first part as desired.

Let $\phi$ be a cut-off function with $\phi(\xi)=1$ if $|\xi|\leq1$, $\phi(\xi)=0$ if $|\xi|>2$, and $0\leq\phi(\xi)\leq1$.
Then it is enough to show that
\begin{equation}\label{135}
\bigg\|\int_{-\infty}^t\int_{\mathbb{R}^n}e^{ix\cdot\xi-i(t-s)|\xi|^2}|\phi(\xi)|^2\widehat{F(\cdot,s)}(\xi)d\xi ds\bigg\|_{L^{q}_tL^{r}_x}
\lesssim \|F\|_{L^{\widetilde{q}'}_tL^{\widetilde{r}'}_x}.
\end{equation}
Once we have this estimate, by the usual scaling we see that for all $j\geq0$
\begin{align*}
\bigg\|\int_{-\infty}^t\int_{\mathbb{R}^n}e^{ix\cdot\xi-i(t-s)|\xi|^2}|\phi(\xi/2^j)|^2&\widehat{F(\cdot,s)}(\xi)d\xi ds\bigg\|_{L^{q}_tL^{r}_x}\\
&\lesssim 2^{-2j(\frac1q+\frac1{\widetilde{q}}-\frac n2(1-\frac1r-\frac1{\widetilde{r}}))}\|F\|_{L^{\widetilde{q}'}_tL^{\widetilde{r}'}_x}\\
&\lesssim \|F\|_{L^{\widetilde{q}'}_tL^{\widetilde{r}'}_x}.
 \end{align*}
Here, for the last inequality, we used the scaling condition \eqref{sca}.
Since we may assume that $F$ is contained in the Schwartz space on $\mathbb{R}^{n+1}$, by a limiting argument ($j\rightarrow\infty$),
we now get \eqref{1:7} from the above estimate.

Now, for fixed $t$, we define
    $$
    T_{t}f(x)
    = \int e^{ix\cdot\xi-it|\xi|^{2}}\phi(\xi)\widehat{f}(\xi)d\xi
    $$
and note that its adjoint operator $T^*_t$ is given by
    $$
    T^*_{t}f(x)=\int e^{ix\cdot\xi+it|\xi|^{2}} \overline{\phi(\xi)} \widehat{f}(\xi)d\xi.
    $$
Then the desired estimate \eqref{135} can be rewritten as
    $$
    \bigg\| \int_{-\infty}^{t} T_t T_s^* F_sds\bigg\|_{L^{q}_tL^{r}_x}
    \lesssim \|F\|_{L^{\widetilde{q}'}_tL^{\widetilde{r}'}_x}
    $$
where we use the notation $F_s$ to denote $F_s(\cdot)=F(\cdot,s)$.
By duality we are now reduced to showing the bilinear form estimate
    \begin{equation}\label{2:1}
    \bigg| \int_{\mathbb{R}} \int_{-\infty}^{t} \langle T^*_s F_s,T^*_t G_t \rangle_{L^2_x} dsdt \bigg|
    \lesssim \|F\|_{L^{\widetilde{q}'}_tL^{\widetilde{r}'}_x}  \|G\|_{L^{q'}_tL^{r'}_x}
    \end{equation}
under the same conditions in Theorem \ref{main_thm}.
To show \eqref{2:1}, we will use the following lemma.

\begin{lem}\label{main_lemma}
Let $n\geq3$, and let $ 2\leq r,\widetilde{r}\leq\infty$ and $1 \leq q, \widetilde{q} \leq \infty$.
Define
    $$
    B_j(F,G)
    =\int_{\mathbb{R}} \int_{t-2^{j+1}}^{t-2^{j}} \langle T_s^* F_s ,T_t^* G_t \rangle_{L_x^2} dsdt
    $$
and assume one of the following conditions for $(r,\widetilde{r};q,\widetilde{q})$:
 $$
    \begin{aligned}
    i)\quad   \frac{1}{r}  \leq \frac{n-2}{n}\frac{1}{\widetilde{r}}
        &\quad\mbox{and}\quad
        \frac{1}{\widetilde{r}} \leq  \frac{1}{q}  \leq  \frac{1}{\widetilde{q}'}  \leq  1  ,\\
    ii)\quad   \frac{n-2}{n}\frac{1}{\widetilde{r}}  \leq \frac{1}{r}  \leq  \frac{1}{\widetilde{r}}
        &\quad\mbox{and}\quad
         -\frac{n}{2}(\frac{1}{r} - \frac{1}{\widetilde{r}}) \leq \frac{1}{q}  \leq \frac{1}{\widetilde{q}'} \leq 1 ,\\
    iii)\quad   \frac{1}{\widetilde{r}}  \leq  \frac{1}{r}  \leq  \frac{n}{n-2}\frac{1}{\widetilde{r}}
        &\quad\mbox{and}\quad
          0  \leq \frac{1}{q} \leq  \frac{1}{\widetilde{q}'} \leq  1 -\frac{n}{2}(\frac{1}{r} - \frac{1}{\widetilde{r}}) ,\\
    iv)\quad   \frac{n}{n-2}\frac{1}{\widetilde{r}} \leq  \frac{1}{r}
        &\quad\mbox{and}\quad
         0 \leq  \frac{1}{q} \leq \frac{1}{\widetilde{q}'}  \leq  1 - \frac{1}{r}.
    \end{aligned}
    $$
Then we have
    \begin{equation}\label{2:6}
    |B_j(F,G)| \lesssim 2^{j \beta(r,\widetilde{r},q,\widetilde{q})}
    \|F\|_{L_t^{\widetilde{q}'}L_x^{\widetilde{r}'}}  \|G\|_{L_t^{q'}L_x^{r'}},
    \end{equation}
where
    $$
    \beta(r,\widetilde{r},q,\widetilde{q})=\left\{\begin{aligned}
    &\frac{1}{q}+\frac{1}{\widetilde{q}}+ \frac{n-1}{\widetilde{r}} -\frac{n}{2} \quad&&\mbox{if\,  i) holds},\\
    &\frac{1}{q}+\frac{1}{\widetilde{q}}-\frac{n}{2}(1-\frac{1}{r}-\frac{1}{\widetilde{r}})\quad&&\mbox{if\,  ii) or iii) holds},\\
    &\frac{1}{q}+\frac{1}{\widetilde{q}}+ \frac{n-1}{r} -\frac{n}{2} \quad&&\mbox{if\,  iv) holds}.
    \end{aligned}\right.
    $$
\end{lem}

\begin{rk}\label{lr}
The ranges of $(1/r,1/\widetilde{r}')$ in $i)$, $ii)$, $iii)$ and $iv)$ correspond to the triangular regions $OB'C'$, $OAB'$, $OAB$ and $OBC$
in Figure \ref{figure}, respectively.
\end{rk}

\begin{proof}[Proof of Lemma \ref{main_lemma}]
One can easily get the above lemma by interpolating the estimates \eqref{2:6} in the following four cases:
\begin{itemize}
\item[(a)] $r=\widetilde{r}=\infty$ (point $O$) \, and \, $1 \leq \widetilde{q}' \leq q \leq \infty$,

\smallskip

\item[(b)] $r=\widetilde{r}=2$ (point $A$) \, and \, $1 \leq \widetilde{q}' \leq q \leq \infty$,

\smallskip

\item[(c)] $r=2$, $\frac{2n}{n-2} \leq\widetilde{r}\leq \infty$ (segment $BC$) \, and \, $2 \leq \widetilde{q}' \leq q \leq \infty$,

\smallskip

\item[(d)] $\frac{2n}{n-2} \leq r\leq \infty$, $\widetilde{r}=2$ (segment $B'C'$) \, and \, $1 \leq \widetilde{q}' \leq q \leq 2$.
\end{itemize}

The first and second ones, $(a)$ and $(b)$, were already shown in \cite{Ko} (see Lemma 2.1 there).
So we only need to show $(c)$ and $(d)$.
For $(c)$ we decompose $F$ and $G$ as
 $$
    F^{k}(x,s)=F(x,s) \chi_{\{2^j k \leq s < 2^j (k+1)\}}(s)\quad\text{and}\quad
    G^{k}(x,t)=G(x,t) \chi_{\{2^j k \leq t < 2^j (k+1)\}}(t)
    $$
for fixed $j$.
Then we see that
    $$
    \begin{aligned}
    |B_j(F,G)|
    &= \sum_{k\in\mathbb{Z}} \int_{\mathbb{R}} \int_{t-2^{j+1}}^{t-2^{j}} \langle T_s^* F_{s}^k ,T_t^* G_{t} \rangle_{L_x^2} dsdt\\
    &\leq \sum_{k\in\mathbb{Z}}\int_{\mathbb{R}} \int_{\mathbb{R}} \langle T_s^* F_{s}^k ,T_t^* G_{t}^{k+1} \rangle_{L_x^2} dsdt
        + \sum_{k\in\mathbb{Z}}\int_{\mathbb{R}} \int_{\mathbb{R}} \langle T_s^* F_{s}^k ,T_t^* G_{t}^{k+2} \rangle_{L_x^2} dsdt
    \end{aligned}
    $$
because $t \in (s+2^{j}, s+2^{j+1})$.
Using H\"older's inequality in $x$, we note that
    $$
    \bigg| \sum_{k\in\mathbb{Z}} \int_{\mathbb{R}} \int_{\mathbb{R}} \langle T_s^* F_{s}^k ,T_t^* G_{t}^{k+1} \rangle_{L_x^2} dsdt\bigg|
    \leq
    \sum_{k \in \mathbb{Z}} \bigg\| \int_{\mathbb{R}} T_s^*F_{s}^{k} ds \bigg\|_{L_x^2}
    \bigg\|\int_{\mathbb{R}}T_t^*G_{t}^{k+1}dt\bigg\|_{L_x^2}.
    $$
We also note that
    \begin{equation}\label{abs}
    \| T_t f \|_{L_t^q L_x^r} \lesssim  \|f\|_{L^2}
    \end{equation}
holds for $r,q \geq 2$ and $n/r+2/q\leq n/2$.
Indeed, by the stationary phase method (see p. 344 in \cite{S}), we see $\| T_t f \|_{L_x^\infty} \lesssim(1+|t|)^{-n/2}\|f\|_{L^1}$.
Then \eqref{abs} follows directly from the abstract Strichartz estimates of Keel and Tao \cite{KT}.
Using the dual estimate of \eqref{abs},
    $$
    \bigg\| \int_{\mathbb{R}} T_t^*F_{s} ds \bigg\|_{L_x^2}
    \lesssim \|F\|_{L^{q'}_{t}L^{r'}_x},
    $$
we now get
    $$
    \bigg| \sum_{k\in\mathbb{Z}} \int_{\mathbb{R}} \int_{\mathbb{R}} \langle T_s^* F_{s}^k ,T_t^* G_{t}^{k+1} \rangle_{L_x^2} dsdt\bigg|
    \lesssim
    \sum_{k \in \mathbb{Z}}  \|F^{k}\|_{L^{2}_{t}L^{\widetilde{r}'}_x} \|G^{k+1}\|_{L^{1}_{t}L^2_x}
    $$
where $\frac{2n}{n-2} \leq \widetilde{r} \leq \infty$.
On the other hand, by H\"older's inequality in time, it follows that
    $$
    \begin{aligned}
    \sum_{k \in \mathbb{Z}}  \|F^{k}\|_{L^{2}_{t}L^{\widetilde{r}'}_x} \|G^{k+1}\|_{L^{1}_{t}L^2_x}
    &\leq
    \sum_{k \in \mathbb{Z}} 2^{j(\frac{1}{q}+\frac{1}{\widetilde{q}}-\frac{1}{2})}
        \|F^k\|_{L^{\widetilde{q}'}_{t}L^{\widetilde{r}'}_x} \|G^{k+1}\|_{L^{q'}_{t}L^2_x}\\
      &\leq
    2^{j(\frac{1}{q}+\frac{1}{\widetilde{q}}-\frac{1}{2})}\|F\|_{L^{\widetilde{q}'}_{t}L^{\widetilde{r}'}_x} \|G\|_{L^{q'}_{t}L^{2}_x}
    \end{aligned}
    $$
if $2 \leq \widetilde{q}' \leq q \leq \infty$.
Here, for the last inequality we used that
    \begin{equation}\label{2:7}
    \sum_{n} |A_n B_n| \leq  \Big( \sum_{n} |A_n|^p \Big)^{\frac{1}{p}} \Big( \sum_{n} |B_n|^{\widetilde{p}} \Big)^{\frac{1}{\widetilde{p}}}
    \quad\mbox{if}\quad \frac{1}{p}+\frac{1}{\widetilde{p}} \geq 1.
    \end{equation}
Hence,
 $$
    \bigg| \sum_{k\in\mathbb{Z}} \int_{\mathbb{R}} \int_{\mathbb{R}} \langle T_s^* F_{s}^k ,T_t^* G_{t}^{k+1} \rangle_{L_x^2} dsdt\bigg|
    \leq
    2^{j(\frac{1}{q}+\frac{1}{\widetilde{q}}-\frac{1}{2})}\|F\|_{L^{\widetilde{q}'}_{t}L^{\widetilde{r}'}_x} \|G\|_{L^{q'}_{t}L^{2}_x}.
    $$
Similarly,
 $$
    \bigg|  \sum_{k\in\mathbb{Z}} \int_{\mathbb{R}} \int_{\mathbb{R}} \langle T_s^* F_{s}^k ,T_t^* G_{t}^{k+2} \rangle_{L_x^2} dsdt\bigg|
    \leq
    2^{j(\frac{1}{q}+\frac{1}{\widetilde{q}}-\frac{1}{2})}\|F\|_{L^{\widetilde{q}'}_{t}L^{\widetilde{r}'}_x} \|G\|_{L^{q'}_{t}L^{2}_x}.
   $$
Consequently, we get $(c)$.
The case $(d)$ can be shown in a similar way as $(c)$.
\end{proof}

Now we return to \eqref{2:1}. It suffices to show that
 \begin{equation}\label{2:00}
    \sum_{j \in \mathbb{Z}} |B_j(F,G)|
    \lesssim \|F\|_{L^{\widetilde{q}'}_tL^{\widetilde{r}'}_x}  \|G\|_{L^{q'}_tL^{r'}_x}.
    \end{equation}
We only consider the case $(1/r,1/\widetilde{r})=P$ since the case $(1/r,1/\widetilde{r})=P'$ follows from the same argument.
Now let
$$(\frac{1}{r},\frac{1}{\widetilde{r}})=(\frac{n-2}{2(n-1)},\frac{(n-2)^2}{2n(n-1)})=P
\quad\text{and}\quad
\frac{n-2}{2(n-1)} \leq \frac{1}{q} = \frac{1}{\widetilde{q}'} < \frac{n}{2(n-1)}.
$$
Note that the point $P$ lies on the segment $OB$, and so we will use Lemma \ref{main_lemma}
under the conditions $iii)$ and $iv)$ (see Figure \ref{figure} and Remark \ref{lr}).
Since $1 -\frac{n}{2}(\frac{1}{r} - \frac{1}{\widetilde{r}})=1-\frac1r=\frac n{2(n-1)}> \frac{1}{\widetilde{q}'}$,
if we choose $\epsilon>0$ small enough so that $\epsilon\leq\frac{1}{20}(\frac n{2(n-1)}-\frac{1}{\widetilde{q}'})$,
we can use Lemma \ref{main_lemma} under the conditions $iii)$ and $iv)$ for $(a,b;q,\widetilde{q})$ with all $(a,b) \in B((\frac{1}{r},\frac{1}{\widetilde{r}}),10\epsilon)$,
where
    $$
    B\big((\frac{1}{r},\frac{1}{\widetilde{r}}),10\epsilon \big)= \{(\frac{1}{a},\frac{1}{b})\in[0,\frac12]\times[0,\frac12]\,:\,|\frac{1}{r}-\frac{1}{a}|, |\frac{1}{\widetilde{r}}-\frac{1}{b}| < 10\epsilon \}.
    $$
Now, using Lemma \ref{main_lemma}, we see that
 \begin{equation}\label{sese}
    |B_j(F,G)|
    \lesssim 2^{j \beta(a,b,q,\widetilde{q})}
        \|F\|_{L_t^{\widetilde{q}'}L_x^{b'}}  \|G\|_{L_t^{q'}L_x^{a'}},
    \end{equation}
 where
    $$
    \beta(a,b,q,\widetilde{q})=\left\{\begin{aligned}
    &\frac{1}{q}+\frac{1}{\widetilde{q}}-\frac{n}{2}(1-\frac{1}{a}-\frac{1}{b})\quad\text{if}\quad \frac{1}{a}\leq\frac{n}{n-2}\frac{1}{b},\\
    &\frac{1}{q}+\frac{1}{\widetilde{q}}+ \frac{n-1}{a} -\frac{n}{2}\quad\text{if}\quad \frac{1}{a}>\frac{n}{n-2}\frac{1}{b}.
    \end{aligned}\right.
    $$
Next we decompose $F$ and $G$ using the following lemma whose proof can be found in \cite{KT} (see Lemma 5.1 there):

\begin{lem}[Atomic decomposition of $L^{p}$]\label{KT_lemma}
Let $1 \leq p<\infty$. Then any $f\in L_{x}^{p}$ can be written as
    $$
    f= \sum_{k=-\infty}^{\infty} c_{k}\chi_{k}
    $$
where each $\chi_{k}$ is a function bounded by $O(2^{-k/p})$
and supported on a set of measure $O(2^{k})$
and the $c_{k}$ are non-negative constants with $\|c_{k}\|_{l^{p}} \lesssim \|f\|_{L^p}$.
\end{lem}

By this lemma, we may write
    $$
    F_s(x)= \sum_{\widetilde{k} \in \mathbb{Z}} f_{\widetilde{k}}(s) \widetilde{\chi}_{\widetilde{k},s}(x)
    \quad\mbox{and}\quad
    G_t(x)= \sum_{k \in \mathbb{Z}} g_{k}(t) \chi_{k,t}(x),
    $$
where $\widetilde{\chi}_{\widetilde{k},s}(x)$ is bounded by
$O(2^{-\widetilde{k}/\widetilde{r}'})$ and supported
on a set of measure $O(2^{\widetilde{k}})$,
and $\chi_{k,t}(x)$ is bounded by $O(2^{-k/r'})$ and
supported on a set of measure $O(2^{k})$.
Also, $f_{\widetilde{k}}$ and $g_{k}$ satisfy
    \begin{equation}\label{2:9-c}
    \Big( \sum_{\widetilde{k} \in \mathbb{Z}} |f_{\widetilde{k}}(s)|^{\widetilde{r}'} \Big)^{\frac{1}{\widetilde{r}'}}
    \lesssim \|F_s\|_{L^{\widetilde{r}'}_x}
    \quad \mbox{and} \quad
    \Big( \sum_{k \in \mathbb{Z}} |g_{k}(t)|^{r'} \Big)^{\frac{1}{r'}}
    \lesssim \|G_t\|_{L^{r'}_x}.
    \end{equation}
Combining \eqref{sese} and this decomposition, we now get
    \begin{equation}\label{part_1}
    \sum_{j\in\mathbb{Z}} |B_j(F,G)|
    \lesssim \sum_{j\in\mathbb{Z}} \sum_{\widetilde{k} \in \mathbb{Z}} \sum_{k \in \mathbb{Z}}
        2^{j \beta(a,b,q,\widetilde{q})} 2^{\widetilde{k}(\frac{1}{\widetilde{r}}-\frac{1}{b})} 2^{k(\frac{1}{r}-\frac{1}{a})}
        \|f_{\widetilde{k}}\|_{L_t^{\widetilde{q}'}} \|g_k\|_{L_t^{q'}}.
    \end{equation}
If $\frac{1}{a}\leq\frac{n}{n-2}\frac{1}{b}$, we use \eqref{part_1}.
But if $\frac{1}{a}>\frac{n}{n-2}\frac{1}{b}$, we use \eqref{part_1} with $2^j$ replaced by $2^{-j}$.
(Since $j \in \mathbb{Z}$, we may replace $2^j$ by  $2^{-j}$ in \eqref{part_1}.)
Then we conclude that
    $$
    \sum_{j \in \mathbb{Z}} |B_j(F,G)|
    \lesssim \sum_{j \in \mathbb{Z}} \sum_{\widetilde{k} \in \mathbb{Z}} \sum_{k \in \mathbb{Z}} H_{j,\widetilde{k}, k}(a,b)
    \|f_{\widetilde{k}}\|_{L_t^{\widetilde{q}'}} \|g_k\|_{L_t^{q'}},
    $$
where
    $$
    \begin{aligned}
    & H_{j,\widetilde{k}, k}(a,b) =
    & \left\{\begin{aligned}
    & 2^{(\widetilde{k}-\frac{n}{2}j)(\frac{1}{\widetilde{r}}-\frac{1}{b}) + (k-\frac{n}{2}j)(\frac{1}{r}-\frac{1}{a})}
        \quad\mbox{if}\quad \frac{1}{a}\leq\frac{n}{n-2}\frac{1}{b},\\
    & 2^{ \widetilde{k}(\frac{1}{\widetilde{r}}-\frac{1}{b}) + (k+(n-1)j)(\frac{1}{r}-\frac{1}{a})}
        \quad\mbox{if}\quad \frac{1}{a}>\frac{n}{n-2}\frac{1}{b}.
    \end{aligned}\right.
    \end{aligned}
    $$

    \begin{figure}[t!]\label{f4}
    \includegraphics[width=7.0cm]{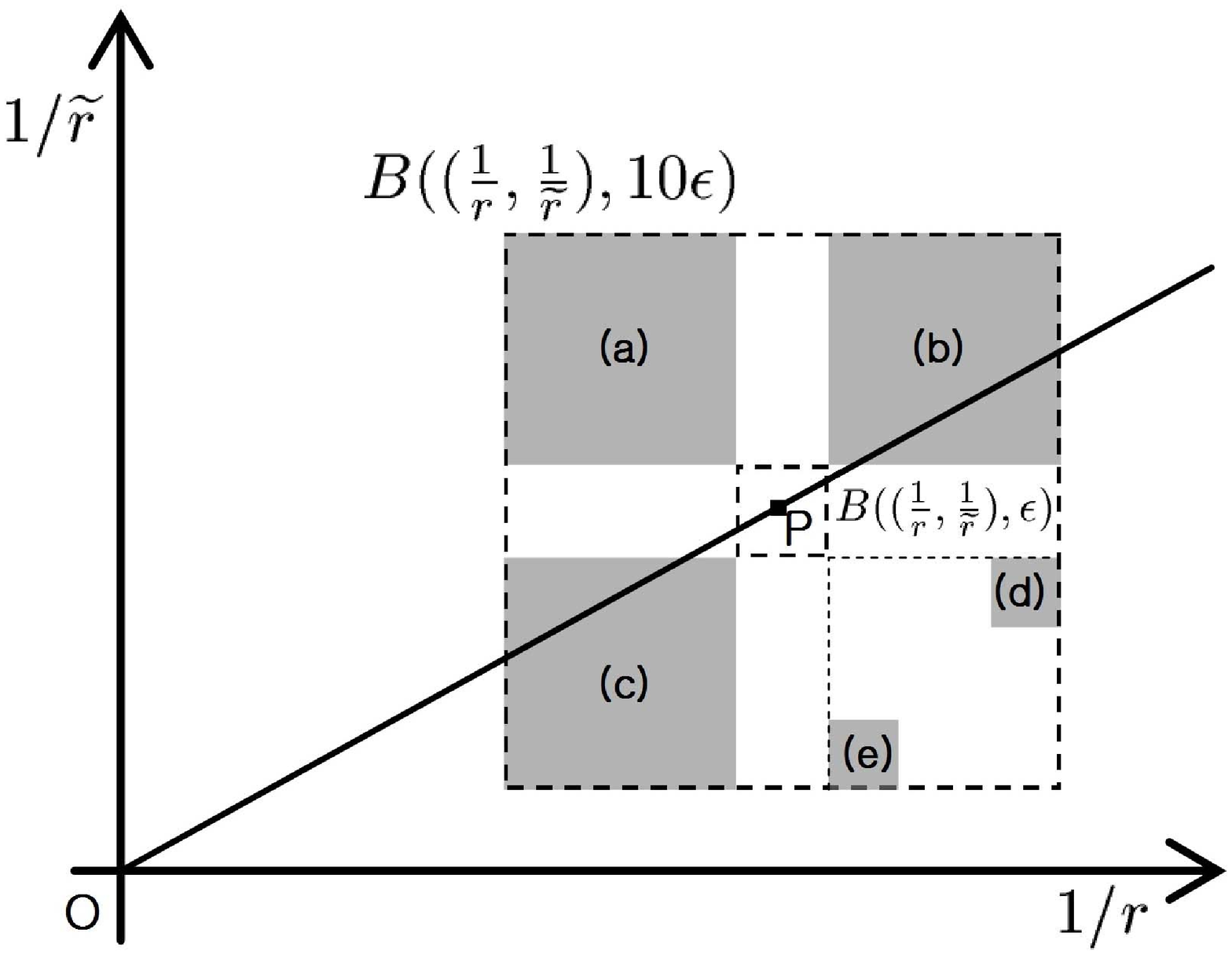}
    \caption{}
    \end{figure}

First we consider the cases where $k\neq\frac{n}{2}j$ and $\widetilde{k}\neq\frac{n}{2}j$.
Let us set
    $$
    \begin{aligned}
    & U_1 = \big\{(j,\widetilde{k},k) ; \quad k-\frac{n}{2}j >0 , \quad \widetilde{k}-\frac{n}{2}j>0 \big\}, \\
    & U_2 = \big\{(j,\widetilde{k},k) ; \quad  k-\frac{n}{2}j <0 , \quad  \widetilde{k}-\frac{n}{2}j>0 \big\}, \\
    & U_3 = \big\{(j,\widetilde{k},k) ;  \quad k-\frac{n}{2}j <0 , \quad  \widetilde{k}-\frac{n}{2}j<0 \big\}, \\
    & U_4 = \big\{(j,\widetilde{k},k) ; \quad  k-\frac{n}{2}j >0 , \quad  \widetilde{k}-\frac{n}{2}j<0 \big\}.
    \end{aligned}
    $$
Then we may write
    \begin{align}\label{2:13}
    \nonumber\sum_{j \in \mathbb{Z}} |B_j(F,G)| &\leq
    \sum_{(j,\widetilde{k},k)\in U_1 \cup U_2 \cup U_3} H_{j,\widetilde{k}, k}(a,b)\|f_{\widetilde{k}}\|_{L_t^{\widetilde{q}'}} \|g_k\|_{L_t^{q'}}\\
    &+ \sum_{(j,\widetilde{k},k)\in U_4} H_{j,\widetilde{k}, k}(a,b)\|f_{\widetilde{k}}\|_{L_t^{\widetilde{q}'}} \|g_k\|_{L_t^{q'}} .
    \end{align}
For each $(j,\widetilde{k},k) \in U_1 \cup U_2 \cup U_3$,
we choose\footnote{The line $OP$ intersects the regions $(b)$ and $(c)$ since its slope is $(n-2)/n$ (see Figure \ref{f4}).
Hence, if $(j,\widetilde{k},k)\in U_1$, choose $(\frac{1}{a},\frac{1}{b})$ that lies above the line $OP$
in the region $(b)$.
If $(j,\widetilde{k},k)\in U_2$, choose $(\frac{1}{a},\frac{1}{b})$ in the region $(a)$.
If $(j,\widetilde{k},k)\in U_3$, choose $(\frac{1}{a},\frac{1}{b})$ that lies above the line $OP$ in the region $(c)$.}
$(\frac{1}{a},\frac{1}{b}) \in B((\frac{1}{r},\frac{1}{\widetilde{r}}),10\epsilon)\setminus B((\frac{1}{r},\frac{1}{\widetilde{r}}),\epsilon) $ with $\frac{1}{a}\leq\frac{n}{n-2}\frac{1}{b}$ so that
    $$
    \begin{aligned}
    \sum_{U_1 \cup U_2 \cup U_3}
    2^{(\widetilde{k}-\frac{n}{2}j)(\frac{1}{\widetilde{r}}-\frac{1}{b}) + (k-\frac{n}{2}j)(\frac{1}{r}-\frac{1}{a})}
    &\|f_{\widetilde{k}}\|_{L_t^{\widetilde{q}'}} \|g_k\|_{L_t^{q'}} \\
\leq
    \sum_{U_1 \cup U_2 \cup U_3}
    &2^{ -\epsilon (|\widetilde{k}-\frac{n}{2}j| + |k-\frac{n}{2}j|)}
    \|f_{\widetilde{k}}\|_{L_t^{\widetilde{q}'}} \|g_k\|_{L_t^{q'}}.
    \end{aligned}
    $$
By summing in $j$, and using \eqref{2:7} and Young's inequality since $K(\cdot) = (1+|\cdot|)2^{-\epsilon|\cdot|}$ is absolutely summable,
we see that
    $$
    \begin{aligned}
\sum_{U_1 \cup U_2 \cup U_3}
    2^{ -\epsilon (|\widetilde{k}-\frac{n}{2}j| + |k-\frac{n}{2}j|)}
    &\|f_{\widetilde{k}}\|_{L_t^{\widetilde{q}'}} \|g_k\|_{L_t^{q'}}\\
\lesssim
    &\sum_{\widetilde{k} \in \mathbb{Z}} \sum_{k \in \mathbb{Z}}
    (1+|\widetilde{k}-k|)2^{-\epsilon |\widetilde{k}-k|}
    \|f_{\widetilde{k}}\|_{L_t^{\widetilde{q}'}} \|g_{k}\|_{L_t^{q'}}\\
\lesssim
     &\Big( \sum_{\widetilde{k} \in \mathbb{Z}}
    \|f_{\widetilde{k}}\|_{L_t^{\widetilde{q}'}}^{\widetilde{q}'} \Big)^{\frac{1}{\widetilde{q}'}}
    \Big( \sum_{\widetilde{k} \in \mathbb{Z}}
    \|g_{\widetilde{k}}\|_{L_t^{q'}}^{q'} \Big)^{\frac{1}{q'}} .
    \end{aligned}
    $$
Since $\widetilde{q}' \geq \widetilde{r}'$ and $q' \geq r'$, by Minkowski's inequality and \eqref{2:9-c},
    \begin{align}\label{labb}
    \nonumber\Big( \sum_{\widetilde{k} \in \mathbb{Z}}
        \|f_{\widetilde{k}}\|_{L_t^{\widetilde{q}'}}^{\widetilde{q}'} \Big)^{\frac{1}{\widetilde{q}'}}
        \Big( \sum_{\widetilde{k} \in \mathbb{Z}}
        \|g_{\widetilde{k}}\|_{L_t^{q'}}^{q'} \Big)^{\frac{1}{q'}}
    \nonumber&\lesssim
    \Big( \sum_{\widetilde{k} \in \mathbb{Z}}
        \|f_{\widetilde{k}}\|_{L_t^{\widetilde{q}'}}^{\widetilde{r}'} \Big)^{\frac{1}{\widetilde{r}'}}
        \Big( \sum_{\widetilde{k} \in \mathbb{Z}}
        \|g_{\widetilde{k}}\|_{L_t^{q'}}^{r'} \Big)^{\frac{1}{r'}} \\
    \nonumber&\lesssim
    \Big\|\Big( \sum_{\widetilde{k} \in \mathbb{Z}}
        |f_{\widetilde{k}}(s)|^{\widetilde{r}'} \Big)^{\frac{1}{\widetilde{r}'}}\Big\|_{L_t^{\widetilde{q}'}}
        \Big\|\Big( \sum_{\widetilde{k} \in \mathbb{Z}}
        |g_{\widetilde{k}}(t)|^{r'} \Big)^{\frac{q'}{r'}}\Big\|_{L_t^{q'}} \\
    &\lesssim
        \|F\|_{L^{\widetilde{q}'}_t L^{\widetilde{r}'}_x} \|G\|_{L^{q'}_t L^{r'}_x}.
    \end{align}
Consequently, we bound the first term in the right-hand side of \eqref{2:13} as desired.
To bound the second term, we first write
    $$
    \begin{aligned}
    \sum_{U_4} H_{j,\widetilde{k}, k }(a,b)
    &= \sum_{j \in \mathbb{Z}} \sum_{\widetilde{k}< \frac{n}{2}j} \sum_{k>\frac{n}{2}j}
        2^{\widetilde{k}(\frac{1}{\widetilde{r}}-\frac{1}{b}) + (k+(n-1)j)(\frac{1}{r}-\frac{1}{a})}
        \|f_{\widetilde{k}}\|_{L_t^{\widetilde{q}'}} \|g_k\|_{L_t^{q'}} \\
    &= \sum_{j \in \mathbb{Z}} 2^{j(n-1)(\frac{1}{r}-\frac{1}{a})}
        \sum_{\widetilde{k}< \frac{n}{2}j} 2^{\widetilde{k}(\frac{1}{\widetilde{r}}-\frac{1}{b}) } \|f_{\widetilde{k}}\|_{L_t^{\widetilde{q}'}}
        \sum_{k>\frac{n}{2}j} 2^{k(\frac{1}{r}-\frac{1}{a})} \|g_k\|_{L_t^{q'}}
    \end{aligned}
    $$
and note that
    $$
    \begin{aligned}
    &\sum_{\widetilde{k} <\frac{n}{2}j} 2^{\widetilde{k}(\frac{1}{\widetilde{r}}-\frac{1}{b}) } \|f_{\widetilde{k}}\|_{L_t^{\widetilde{q}'}}
        \sum_{k>\frac{n}{2}j} 2^{k(\frac{1}{r}-\frac{1}{a})} \|g_k\|_{L_t^{q'}} \\
    &\quad\leq \Big( \sum_{\widetilde{k} <\frac{n}{2}j}
        2^{\widetilde{k}(\frac{1}{\widetilde{r}}-\frac{1}{b})\widetilde{q}}\Big)^{\frac{1}{\widetilde{q}}}
        \Big( \sum_{\widetilde{k}<\frac{n}{2}j} \|f_{\widetilde{k}}\|_{L_t^{\widetilde{q}'}}^{\widetilde{q}'} \Big)^{\frac{1}{\widetilde{q}'}}
        \Big( \sum_{k>\frac{n}{2}j} 2^{k(\frac{1}{r}-\frac{1}{a})q} \Big)^{\frac{1}{q}}
        \Big( \sum_{k>\frac{n}{2}j}\|g_k\|_{L_t^{q'}}^{q'} \Big)^{\frac{1}{q'}} \\
    &\quad\lesssim2^{ \frac{n}{2}j (\frac{1}{\widetilde{r}}-\frac{1}{b}) } 2^{ \frac{n}{2}j (\frac{1}{r}-\frac{1}{a})}
        \|F\|_{L^{\widetilde{q}'}_t L^{\widetilde{r}'}_x} \|G\|_{L^{q'}_t L^{r'}_x}.
    \end{aligned}
    $$
Here we used \eqref{labb} for the last inequality.
Hence we get
    $$
    \sum_{U_4} H_{j,\widetilde{k}, k }(a,b)
    \lesssim\sum_{j \in \mathbb{Z}}
        2^{ \frac{n}{2}j (\frac{1}{\widetilde{r}}-\frac{1}{b}) }
        2^{ \frac{3n-2}{2}j (\frac{1}{r}-\frac{1}{a})}
        \|F\|_{L^{\widetilde{q}'}_t L^{\widetilde{r}'}_x} \|G\|_{L^{q'}_t L^{r'}_x}.
    $$
Now we choose\footnote{Choose $(\frac{1}{a},\frac{1}{b})$ in the region $(d)$ when $j\geq0$
and in the region $(e)$ when $j<0$ (see Figure \ref{f4}).}
$(\frac{1}{a},\frac{1}{b}) \in B((\frac{1}{r},\frac{1}{\widetilde{r}}),10\epsilon)\setminus B((\frac{1}{r},\frac{1}{\widetilde{r}}),\epsilon) $ with
$$\frac{1}{a}>\frac{n}{n-2}\frac{1}{b},
\quad -10\epsilon < \frac{1}{r}-\frac{1}{a} < -9\epsilon,
\quad 2\epsilon > \frac{1}{\widetilde{r}}-\frac{1}{b} > \epsilon$$
when $j\geq0$,
and with
$$\frac{1}{a}>\frac{n}{n-2}\frac{1}{b},
\quad-2\epsilon < \frac{1}{r}-\frac{1}{a} < -\epsilon,
\quad  10\epsilon > \frac{1}{\widetilde{r}}-\frac{1}{b} > 9\epsilon
$$
when $j<0$.
Then we get the desired bound
    $$
     \sum_{U_4} H_{j,\widetilde{k}, k}(a,b)
    \lesssim\|F\|_{L^{\widetilde{q}'}_t L^{\widetilde{r}'}_x} \|G\|_{L^{q'}_t L^{r'}_x}.
    $$
Consequently, we get \eqref{2:00}.

\medskip

Finally we consider the cases where $k=\frac{n}{2}j$ or $\widetilde{k}=\frac{n}{2}j$.
When $k=\frac{n}{2}j$, we note that
    $$
    \sum_{j \in \mathbb{Z}} \sum_{\widetilde{k} \in \mathbb{Z}} \sum_{k \in \mathbb{Z}} H_{j,\widetilde{k}, k}(a,b)
    \|f_{\widetilde{k}}\|_{L_t^{\widetilde{q}'}} \|g_k\|_{L_t^{q'}}
    = \sum_{\widetilde{k} \in \mathbb{Z}} \sum_{k \in \mathbb{Z}} 2^{(\widetilde{k}-k)(\frac{1}{\widetilde{r}}-\frac{1}{b})}
    \|f_{\widetilde{k}}\|_{L_t^{\widetilde{q}'}} \|g_k\|_{L_t^{q'}},
    $$
where $\frac{1}{a}\leq\frac{n}{n-2}\frac{1}{b}$.
Hence if we choose $(a,b)$ in the region $(a)$ or $(c)$ with $\frac{1}{a}\leq\frac{n}{n-2}\frac{1}{b}$, we see
    $$
    \sum_{\widetilde{k} \in \mathbb{Z}} \sum_{k \in \mathbb{Z}} 2^{(\widetilde{k}-k)(\frac{1}{\widetilde{r}}-\frac{1}{b})}
    \|f_{\widetilde{k}}\|_{L_t^{\widetilde{q}'}} \|g_k\|_{L_t^{q'}}
    \leq \sum_{\widetilde{k} \in \mathbb{Z}} \sum_{k \in \mathbb{Z}} 2^{-\epsilon|\widetilde{k}-k|}
    \|f_{\widetilde{k}}\|_{L_t^{\widetilde{q}'}} \|g_k\|_{L_t^{q'}} .
    $$
From this, we get the desired bound as before. The other cases follow easily in a similar way.

\section{Necessary conditions}\label{sec3}

It was shown in \cite{F} that
\begin{equation}\label{necess}
\frac{n-2}{r} - \frac{2}{q} \leq \frac{n}{\widetilde{r}}\quad\text{and}\quad\frac{n-2}{\widetilde{r}} - \frac{2}{\widetilde{q}} \leq \frac{n}{r}
\end{equation}
are the necessary conditions for which the inhomogeneous estimate
\begin{equation}\label{1:45}
\bigg\|\int_{0}^{t}e^{i(t-s)\Delta}F(\cdot,s)ds\bigg\|_{L^{q}_tL^{r}_x} \lesssim \|F\|_{L^{\widetilde{q}'}_tL^{\widetilde{r}'}_x}
\end{equation}
holds.
(We refer the reader to \cite{F,V,LS} for other necessary conditions.)
Compared with \eqref{necess}, we give here the following new necessary condition:
\begin{equation}\label{necess2}
\frac{n-2}{\widetilde{r}} - \frac{2}{q} \leq \frac{n}{r},\quad\frac{n-2}{r} - \frac{2}{\widetilde{q}} \leq \frac{n}{\widetilde{r}}.
\end{equation}
The first condition in \eqref{necess2} is stronger than the first one in \eqref{necess} when $1/r\leq1/\widetilde{r}$,
and the second condition in \eqref{necess2} is stronger than the second one in \eqref{necess} when $1/r\geq1/\widetilde{r}$.

\begin{proof}[Proof of \eqref{necess2}]
If \eqref{1:45} holds with a pair $(r,q)$ on the left and a pair $(\widetilde{r},\widetilde{q})$
on the right, then it must be also valid when one switches the roles of $(r,q)$ and $(\widetilde{r},\widetilde{q})$.
By this duality relation, we only need to show the second condition $(n-2)/r - 2/\widetilde{q} \leq n/\widetilde{r}$ in \eqref{necess2}.
For this, we first write
\begin{equation*}
    I(F)
    := \int_{0}^t e^{i(t-s)\Delta} F (\cdot,s)ds
    = (4 \pi)^{- \frac{n}{2}} \int_0^t \int_{\mathbb{R}^n} |t-s|^{- \frac{n}{2}} e^{i \frac{|x-y|^2}{4|t-s|}} F(y,s) dy ds.
    \end{equation*}
Let $0 < \epsilon < 1/2$ and
$F(y,s)=\chi_{\{0<s<\epsilon^2, |y|<\epsilon\}}$.
Since
    $$
    \frac{|x-y|^2}{4(t-s)}
    = \frac{|x|^2}{4t} + \frac{t(|x-y|^2 -|x|^2)+ s|x|^2}{4t(t-s)},
    $$
for $0<s<\epsilon^2$, $|y|<\epsilon $, $10 < t < 11$ and $\big| |x| - \frac{1}{\epsilon} \big| < \epsilon$,
we see that
    $$
    \Big| \frac{t(|x-y|^2 -|x|^2)+ s|x|^2}{4t|t-s|} \Big|
    < \frac{ 11\cdot3 +\epsilon^2\cdot\frac{2}{\epsilon^2} }{4\cdot10\cdot9}
    < \frac{1}{2}.
    $$
Hence, if $10 < t < 11$ and $\big| |x| - \frac{1}{\epsilon} \big| < \epsilon$,
    \begin{align*}
    | I (F) (x,t) |
    &\geq \bigg|(4 \pi)^{- \frac{n}{2}} e^{i\frac{|x|^2}{4t}}
        \int_0^t \int |t-s|^{- \frac{n}{2}} e^{i\frac{t(|x-y|^2 -|x|^2)+ s|x|^2}{4t(t-s)}} F(y,s) dy ds \bigg| \\
    &\gtrsim\int_0^{\epsilon^2} \int_{|y|<\epsilon} dy ds\\
    &\gtrsim \epsilon^{n+2}.
    \end{align*}
This implies that
    \begin{align*}
    \|I(F)\|_{L^{q}_{t}L^{r}_{x}}
    &\geq \|I(F)\|_{L^{q}_{t}(10 < t < 11)L^{r}_{x}(||x|-\frac{1}{\epsilon}|<\epsilon)}\\
     &\gtrsim\epsilon^{n+2}\big((1/\epsilon+\epsilon)^n - (1/\epsilon-\epsilon)^n\big)^{1/r}\\
     &\gtrsim\epsilon^{n+2}\big(\epsilon^{-n}( (1+\epsilon^2)^n - (1-\epsilon^2)^n)\big)^{1/r}\\
    &\gtrsim \epsilon^{n+2} \epsilon^{\frac{-n+2}{r}}.
    \end{align*}
On the other hand,
$\|F\|_{L^{\widetilde{q}'}_{s}L^{\widetilde{r}'}_{y}} \sim \epsilon^{\frac{2}{\widetilde{q}'}} \epsilon^{\frac{n}{\widetilde{r}'}}$.
Now the estimate \eqref{1:45} leads us to
$\epsilon^{n+2} \epsilon^{\frac{-n+2}{r}} \lesssim  \epsilon^{\frac{2}{\widetilde{q}'}} \epsilon^{\frac{n}{\widetilde{r}'}}$.
By letting $\epsilon \rightarrow0$, we conclude that $2/\widetilde{q}+ n/\widetilde{r}\geq (n-2)/r$.
\end{proof}

%%%%%%%%%%%%%%%%%%%%%%%%%%%%%%%%%%%%%%%%%%%%%%%%%%%%%%%%%%%%%%%%%%%%%%%%%%%%%%%%%%%%%%%%%%%%5

\end{document}